\documentclass[aop, preprint]{imsart}
\RequirePackage[OT1]{fontenc}
\RequirePackage{amsthm,amsmath,natbib}
\RequirePackage[colorlinks]{hyperref}
\RequirePackage{hypernat}

\numberwithin{equation}{section}
\usepackage{verbatim}
\usepackage{amssymb}
\usepackage{comment}
\usepackage{appendix}
\usepackage{graphicx}
\usepackage{subfigure}
\usepackage{enumerate}
\usepackage{comment}
\usepackage{pdfsync}
\usepackage{microtype}
\usepackage{subfigure}
\usepackage{color} % Need the color package
\usepackage[table]{xcolor}
\usepackage{todonotes}
\usepackage{bbm, dsfont}
\usepackage{natbib}
\usepackage[left=3cm, right=3cm]{geometry}
\input xy
\xyoption{all}

\startlocaldefs

\newtheorem{theorem}{Theorem}[section]
\newtheorem{lemma}[theorem]{Lemma}
\newtheorem{proposition}[theorem]{Proposition}
\newtheorem{corollary}[theorem]{Corollary}

\newtheorem{definition}[theorem]{Definition}

\newcommand{\Ev}{\mathbb{E}}
\newcommand{\E}{\mathbb{E}}

\newcommand{\Pv}{\mathbb{P}}

\newcommand{\CE}{\mathcal {E}}

\newcommand{\CL}{\mathcal {L}}

\newcommand*{\be}{\begin{equation}}
\newcommand*{\ee}{\end{equation}}
\newcommand*{\ba}{\begin{aligned}}
\newcommand*{\ea}{\end{aligned}}
\newcommand*{\barr}{\begin{array}{c}}
\newcommand*{\earr}{\end{array}}

\newcommand{\BIN}{{\sf Bin}}
\newcommand{\CMnD}{\mathrm{CM}_n(\boldsymbol{D})}

\newcommand{\eqn}[1]{\begin{equation}#1\end{equation}}
\newcommand{\eqan}[1]{\begin{align}#1\end{align}}
\newcommand{\nn}{\nonumber}
\newcommand{\sss}{\scriptscriptstyle}
\newcommand{\e}{\mathrm{e}}

\begin{document}

\begin{frontmatter}
\title{First passage percolation on random graphs with infinite variance degrees}
\runtitle{FPP on random graphs with infinite variance degrees }

\begin{aug}

\runauthor{E. Baroni, R. van der Hofstad, J. Komj\'athy}
\author{Enrico Baroni}
\address{Eindhoven University of Technology,\protect\\
\tt{e.baroni@tue.nl}}

\author{Remco van der Hofstad}
\address{Eindhoven University of Technology,\protect\\
\tt{r.w.v.d.hofstad@TUE.nl}}

\author{J\'{u}lia Komj\'athy}
\address{Eindhoven University of Technology,\protect\\
\tt{j.komjathy@tue.nl}}

\affiliation{Eindhoven University of Technology}

\end{aug}
\date{\today}

\begin{abstract}
We prove non-universality results for first-passage percolation on the configuration model with i.i.d.\ degrees having infinite variance. We focus on the weight of the optimal path between two uniform vertices. Depending on the properties of the weight distribution, we use an example-based approach and show that rather different behaviors are possible. When the weights are a.s.\ larger than a constant, the weight and number of edges in the graph grow proportionally to $\log\log{n}$, as for the graph distances. On the other hand, when the continuous-time branching process describing the first passage percolation exploration through the graph reaches infinitely many vertices in finite time, the weight converges to the sum of two i.i.d.\ random variables representing the explosion times of the processes started from the two sources. This non-universality is in sharp contrast to the setting where the degree sequence has a finite variance \cite{BHH13}.
\end{abstract}

\begin{keyword}[class=AMS]
\kwd[Primary ]{60J10} %markov chains
\kwd{60D05}
\kwd{37A25}
\end{keyword}

\begin{keyword}
\kwd{Configuration model, power-law degrees, first passage percolation, universality}
\end{keyword}

\end{frontmatter}

\maketitle

\section{Introduction}

\subsection{The model}
There is a wide literature on how complex networks appear in many different situations as models for, e.g., communication networks \cite{wu}, gene regulatory networks \cite{Bornholdt:2003}, electric power grids \cite{2004PhRvE..69b5103A}, the Internet \cite{Faloutsos}, transportation systems \cite{2010arXiv1001.2172K}, the World-Wide Web \cite{Barabási200069}, citation networks \cite{citation} and social networks \cite{RevModPhys, 2003SIAMR}. All these examples deal with systems composed of many highly connected units. We use graphs denoted by $G$ to model such settings. The first mathematical model of such networks is the Erd{\H o}s-R\'enyi graph (see \cite{erdHos1976evolution}), in which there is a link between any pair of nodes independently with a certain probability, so the nodes have the same number of links on average. In recent years, however (mainly due to a large amount of available data), two characteristics for many of these networks have been observed. The first is the so-called ``small world phenomenon'', meaning that the graph distances between vertices are small. The second observed phenomenon is that some of the nodes have a large amount of links, and in particular the degrees of the vertices is close to a power-law distribution (see for instance \cite{Faloutsos}). One of the most studied properties of such, so-called {\em scale-free} random graphs, is the spread of information on them.

In particular, we are interested in the following setting: we have a transportation network that transports a flow through the edges of the network. Basically, the behaviour of the flow depends on two factors: the number of edges in the shortest path between the vertices of the network and the passage cost through the edges of it. From a mathematical point of view, this leads us to study distances on random graphs having infinite-variance degrees using first-passage percolation, where the weights on edges are given by independent and identically distributed (i.i.d.) random variables.

First-passage percolation was first introduced by Hammersley and Welsh (see \cite{MR0483050}), and is obtained by assigning the collection of i.i.d.\ random weights $(Y_{e})_{e\in E}$ to the edges $E$ of the graph $G$. If $u$ is a fixed vertex in $G$, then we define the \emph{passage-time} from $u$ to a vertex $v$ as
 	\be\label{passagetime}
	W_{n}(u,v):=\min_{\pi\colon u\rightarrow v}\sum_{e\in \pi}Y_{e},
	\ee
where the minimum is taken over all paths of the form $\pi\subset E $ from $u$ to $v$. Here, we set $W_{n}(u,v)=0$ when $u=v$, while $W_n(u,v)=\infty$ when $u$ and $v$ are not connected.

The main object of study in first-passage percolation is the ball of radius $t$ in the $W_n$-metric given by
	\be
	B(t)=\{v\in G\colon W_{n}(u,v)\leq t\}.
	\ee
More precisely, in this paper we study the configuration model $\CMnD$ on $n$ vertices, where each vertex has a number of half-edges given by i.i.d. random variables $(D_v)_{v\in [n]}$, where $[n]=\{1,\dots n\}$ distributed as $D$, with a power-law distribution that satisfies
	\be
	\label{eq::F}
	\frac{c_1}{x^{\tau-1}}\le 1- F_{\sss D}(x)= \Pv(D>x) \le \frac{C_1}{x^{\tau-1}}.
	\ee
We further assume that
	\be
	\label{gc}
	\Pv(D\geq 2)=1,
	\ee
and that $\tau \in (2,3)$, so that the degrees have finite mean, but infinite variance.
We pair these half-edges uniformly at random and without replacement. Condition \eqref{gc} guarantees that almost all the vertices of the graph lie in the same connected component, or, equivalently, the giant component has size close to $n$ (see \cite[Proposition 2.1]{KHB13}).
All edges are equipped with i.i.d.\ weights or lengths with distribution function $F_{\sss Y}(y)=\Pv(Y\le y)$, where $Y$ is a non-negative random variable having a continous distribution.
We assume $ \Pv(Y\geq 0)=1$. Let $B$ defined as the random variable $D^{\star}-1$ through the size-biased distribution of the variable $D^{\star}$, i.e.,
	\begin{align}\label{eq::sizebiased}
    	&\Pv(D^{\star}=k):=\frac{k}{\E[D]}\Pv(D=k),
	\end{align}
and
	\begin{align}
	&\Pv(B=k):=\Pv(D^{\star}=k+1)=\frac{k+1}{\E[D]}\Pv(D=k+1).
	\end{align}
Then, with $F_{\sss B}(x):=\Pv(B\leq x)$, there exist constants $0<c_{1}^{\star}<C_{1}^{\star}<\infty$ such that (see e.g., \cite[Theorem 3.1]{KHB13})
	\be
	\label{eq::F2} \frac{c^{\star}_1}{x^{\tau-2}}\le 1- F_{\sss B}(x)= \Pv(B>x) \le \frac{C^{\star}_1}{x^{\tau-2}}.
	\ee
Let $h_{\sss B}(s)=\sum_{k=1}^{\infty}\Pv(B=k)s^{j}$  be the probability generating function of $B$.

\paragraph{Notation}
We use the following standard concepts and notation. We say that a sequence of random variables converges in probability to a random variable $X$, and we write $X_{n}\overset{\Pv}{\rightarrow}X$ if, for every $\varepsilon>0$, $\Pv(|X_{n}-X|>\varepsilon)\rightarrow 0$. Further, we say that $X_{n}$ converges to $X$ in distribution, and we write $X_{n}\overset{d}{\rightarrow}X,$ if $\lim_{n\rightarrow\infty}\Pv(X_{n}\leq x)=\Pv(X\leq x)$ for every $x$ for which $F_{X}(x)=\Pv(X\leq x)$ is continous.
Finally we say that a sequence $E_{n}$ of events holds w.h.p.\ (with high probability) if $\lim_{n\rightarrow\infty}\Pv(E_{n})=1$.

The main tool in the proof of typical distances is a connection to continuous-time (age-dependent) branching processes:

\begin{definition}[Age-dependent branching process]
We call a branching process \emph{age-dependent}, if individuals have random life-lengths with distribution function $F_{\sss Y}$ with $F_{\sss Y}(0)=0$. At death, an individual produces offspring of random size with probability generating function $h(s)$ and all life-lengths and family sizes are independent of each other. We assume that the process starts at time $t=0$ with one individual of age 0.
\end{definition}

\subsection{Previous results}
We investigate the total weight of the shortest-weight path between two uniformly chosen vertices, as in \eqref{passagetime}. When $Y$ has exponential distribution with mean $1$, then, in \cite[Theorem 3.2]{BHH10}, the following result was proved for $W_{n}(u,v)$:

\begin{theorem}[\cite{BHH10}]\label{dist}
\label{thm-expon}
Consider the configuration model $\CMnD$ where the degrees are i.i.d.\ with distribution function $F_{\sss D}$ satisfying \eqref{eq::F}, and with i.i.d.\ edge weights $Y$ having exponential distribution with mean $1$. Then, the weight $W_{n}(u,v)$ of the shortest-weight path between two uniformly chosen vertices $u$ and $v$ satisfies
	\be
	W_{n}(u,v)\overset{d}{\to} V^{\sss(1)}+V^{\sss(2)},
	\ee
where $V^{\sss(1)},V^{\sss(2)}$ are independent copies of the explosion time of a continous-time age-dependent branching process with infinite-mean offspring distribution that we define below.
\end{theorem}

The first of our main results extends Theorem \ref{thm-expon} to a more general family of edge weights. For this, let $(h_{\sss B},F_{\sss Y})$ be a modified age-dependent branching process where individuals have random life-lengths with distribution $F_{\sss Y}(t)$ and, at death, the first individual (the root) produces a family of random size with offspring distribution $F_{\sss D}$ and the further generations have offspring distribution $F_{\sss B}$. Recall that $h_{\sss B}(s)$ denotes the probability generating function of the distribution  $F_{\sss B}$.

\begin{definition}[Explosive age-dependent process]
We say that the branching process $(h_{\sss B},F_{\sss Y})$ is \emph{explosive} if there is a positive probability that $N_{t}=\infty$, where $N_{t}$ denotes the number of individuals alive at some finite time $t>0$. Otherwise, it is called {\em conservative}.
\end{definition}

The following theorem from \cite{Grey74} shows that for every offspring distribution $F_{\sss B}$ with infinite expectation, there is a weight distribution $F_{\sss Y}$ for which the process is explosive:

\begin{theorem}[{\protect \cite[Theorem 6]{Grey74}}]
Given $h_{\sss B}$ such that $h'_{\sss B}(1)=\infty$, there exists a $F_{\sss Y}$ with $F_{\sss Y}(0)=0$ such that the process $(h_{\sss B},F_{\sss Y})$ is explosive.
\end{theorem}

%\begin{proposition}[\cite{Grey74}]\label{1.5}
%If $h_{\sss B}(s)$ is the p.g.f.\ of the distribution in \eqref{eq::F2}, then the process $(h_{\sss B},F_{\sss Y})$ is explosive if and only if the process $(\widetilde{h},F_{\sss Y})$ is explosive, where $\widetilde{h}(s)=1-(1-s)^{\tau-2}$.
%\end{proposition}

The next theorem \cite[Corollary 4.1]{Grey74}, compares the branching process $(h_{\sss B},F_{\sss Y})$ to a Markov branching process:

\begin{theorem}[Comparison to Markov branching processes] \label{mainth}
If there exists $T>0$ and constants $\beta>\alpha>0$ such that $\alpha t\leq F_{\sss Y}(t)\leq \beta t$ for $0\leq t\leq T$, then the process $(h_{\sss B},F_{\sss Y})$ is explosive if and only if
$h'_{\sss B}(1)=\infty$ and
	\be
	\label{cond}
 	\int_{1-\varepsilon}^{1}\frac{ds}{s-h(s)}<\infty
	\ee
for some $\varepsilon>0$.
\end{theorem}

The condition in \eqref{cond} is necessary and sufficient for explosiveness for Markov branching processes, as proved by Harris in \cite{Harr63}.

\subsection{Our results}

The first result of our paper is the following theorem:

\begin{theorem}[Universality class for explosive weights]
\label{thm::main1}
Consider the configuration model with i.i.d.\ degrees having distribution function $F_{\sss D}$ satisfying \eqref{eq::F}, and i.i.d.\ edge weights $(Y_e)_{e\in E}$ having distribution $F_{\sss Y}$.
Further suppose that $(h_{\sss B},F_{\sss Y})$ is explosive. If $u$ and $v$ are two uniformly chosen vertices, then
 	\begin{equation}
	W_{n}(u,v)\overset{d}{\to}V^{\sss{(1)}}+V^{\sss{(2)}},
	\end{equation}
where $V^{\sss(1)}$ and $V^{\sss(2)}$ are two i.i.d.\ copies of the explosion time of the process $(h_{\sss B},F_{\sss Y})$ defined above.
\end{theorem}
\medskip

Theorem \ref{thm::main1} implies that all edge weights for which the age-dependent branching process that approximates the local neighborhoods of vertices is explosive are in the same universality class. We next investigate one class of random edge weights that are in a different universality class. For this, it is useful to define the hopcount as the number of edges in the shortest-weight path:

\begin{definition}\label{Hopcount}
The \emph{hopcount} $H_{n}(u,v)$ is the number of edges in the shortest-weight path between $u$ and $v$.
\end{definition}

The next result determines the asymptotic behaviour of the weight and the hopcount in a different setting. We now consider the case when the weight is given by $Y=c+X$ such that $X$ is a random variable with $\inf\mathrm{supp}(X)=0$, where $c$ is a constant satisfying $c>0$ and $\mathrm{supp}(X)$ is the support of the distribution. Without loss of generality, we may assume that $c=1$.
Then, our result in this setting is the following theorem:

\begin{theorem}[Universality class for weights with $\inf\mathrm{supp}(X)>0$]
\label{thm::main2}
Consider the configuration model with i.i.d.\ degrees from distribution $F_{\sss D}$ satisfying \eqref{eq::F}, and i.i.d.\ edge weights
	\be
	\label{peso}
	Y=1+X,
	\ee
 where $\inf\mathrm{supp}(X)=0$. Then,
	\be\ba
	\frac{W_{n}(u,v)}{\log\log{n}}\buildrel{d}\over{\longrightarrow} \frac{2}{|\log{(\tau-2)}|},\quad \quad \frac{H_{n}(u,v)}{\log\log{n}}        				\buildrel{d}\over{\longrightarrow} \frac{2}{|\log{(\tau-2})|}.
	\ea\ee
\end{theorem}
\medskip

Let us comment on the result in Theorem \ref{thm::main2}. Define the graph distance $\mathcal{D}_{n}(u,v)$ between two vertices $u,v \in [n]$ as the minimal number of edges on a path connecting $u$ and $v$. In \cite[Theorem 3.1]{MR2318408}, it was shown that
	\be
	\label{conv}
	\ba
	\frac{\mathcal{D}_{n}(u,v)}{\log\log{n}}\buildrel{\Pv}\over{\longrightarrow}
	\frac{2}{|\log{(\tau-2)}|}.
	\ea\ee
Then, Theorem \ref{thm::main2} shows that the shortest path in terms of its number of edges is such that the average \emph{additional} weight compared to the graph distance vanishes. In other words, there must exist an almost shortest path for which the sum of weights is $o(\log\log{n})$.

\subsection{Overview of the proofs of Theorems \ref{thm::main1}--\ref{thm::main2}}\label{sec14}
In this section we present an overview of the proof of our main results.
\paragraph{Overview of the proof of Theorem \ref{thm::main1}} The key ideas in the proof of Theorem \ref{thm::main1} are the following:

\begin{enumerate}
\item \label{ph::bp}
\emph{Perfect coupling to an age-dependent branching process and lower bound on the weight.}
First we couple the growth of two shortest-weight graphs from $u$ and $v$ respectively, here denoted as ${\sf SWG}^{u}_{n^{\rho}}$ and ${\sf SWG}^{v}_{n^{\rho}}$,  where $u$ and $v$ are uniformly chosen vertices, to two age-dependent branching processes. This can be successfully performed until the time that they reach size $n^{\rho}$, where $\rho$ is a small constant. With high probability these two graphs representing the local neighborhoods of $u$ and $v$, are disjoint trees (see \cite[Proposition 4.7]{BHH10}).

The lower bound in Theorem \ref{thm::main1} follows immediately from the coupling because the two shortest-weight graphs are with high probability disjoint.
We define $N^{\sss(1)}_{t}$ as  the number of dead individuals in the age-dependent branching process associated to ${\sf SWG}^{v}_{n^{\rho}}$ (or ${\sf SWG}^{u}_{n^{\rho}}$) at time $t$, and
	 \be
	V^{\sss(1)}_{m}=\min\{t \colon |N_{t}^{\sss(1)}|=m\}.
	\ee
Also, if $n$ is large enough, $V_{m}^{\sss(1)}$ is the time for the ${\sf SWG}_{m}^{v}$ to reach the $m$th vertex for any $m<n^{\rho},$ where $\rho>0$ is a small constant.
If $V^{\sss(1)}_{\infty}=\min\{t \colon |N^{\sss(1)}_{t}|=\infty \}$, then it coincides with the explosion time of the process that we call $V^{\sss(1)}$.

\item \label{bof}
\emph{Upper bound on the weights}

To prove the upper bound $W_{n}(u,v) \leq V_{n^{\rho}}^{\sss(1)}+V_{n^{\rho}}^{\sss(2)}+\varepsilon$ where $\varepsilon>0$ is some arbitrary small constant, we show that whp there exists a path that connects ${\sf SWG}^{v}_{n^{\rho}}$  to ${\sf SWG}^{u}_{n^{\rho}}$ with weight at most $\varepsilon$. Then we use that $(V_{n^{\rho}}^{\sss(1)},V_{n^{\rho}}^{\sss(2)})\overset{d}{\to}(V^{\sss(1)},V^{\sss(2)})$. We find this path using percolation on the so-called \emph{core} of the graph consisting of vertices of large degrees. In more detail, we only keep edges with weight less than $\varepsilon'$ and then show that there still remains a path via the high degree vertices from ${\sf SWG}^{v}_{n^{\rho}}$ to ${\sf SWG}^{u}_{n^{\rho}}$ that has a number of edges not depending on $n$. Picking $\varepsilon'$ small enough finishes the proof.
\end{enumerate}

\paragraph{Overview of the proof of Theorem \ref{thm::main2}}

The proof consists in verifying that, with high probability and for all $\varepsilon>0$,
	\be
	\label{uineq2}
	\frac{2(1-\varepsilon)\log\log n}{|\log(\tau-2)|}\leq \mathcal{D}_{n}(u,v)\leq H_{n}(u,v)\leq W_{n}(u,v)\leq \frac{2(1+\varepsilon)}{|\log(\tau-2)|}\log\log n.
	\ee
	
The first inequality has been shown in \cite[Theorem 1.2]{MR2318408}, the second and the third are obvious since every path with $q$ edges has weight greater than $q$ for our choice of $Y$. It remains to prove the upper bound on $W_{n}(u,v)$:

\begin{proposition}[Upper bound on weight when $\inf\mathrm{supp}(X)>0$]
\label{prop-Wn-UB}
Consider the configuration model with i.i.d.\ degrees from distribution $F_{\sss D}$ satisfying \eqref{eq::F}, and i.i.d.\ edge weights $Y=1+X$ with $\inf\mathrm{supp}(X)=0$. If $u$ and $v$ are two uniformly chosen vertices in $[n]$, then, for any fixed $\varepsilon>0$ and with high probability as $n\rightarrow\infty$,
	\be
	\frac{W_{n}(u,v)}{\log \log n}\leq \frac{2(1+\varepsilon)}{|\log(\tau-2)|}.
	\ee
\end{proposition}
Theorem \ref{thm::main2} follows directly from Proposition \ref{prop-Wn-UB} and (\ref{uineq2}).

\subsection{Discussions and related problems}

We can consider different metrics and topologies on the same graph given by various notions of distances, an example being the graph distance $\mathcal{D}_{n}(u,v)$ between two vertices $u$ and $v$.  Another metric is defined by $W_{n}(u,v)$, while a third one is characterized by the hopcount $H_{n}(u,v)$. In this last case we deal with a so-called {\em pseudometric}, since the triangle inequality does not necessarily hold. A natural question is whether these three metric spaces are similar. The result of this paper is that this is not the case for a large class of graphs, namely, the configuration model with power-law degree exponent $\tau\in(2,3)$ and i.i.d.\ edge weights $Y=1+X$ with $\inf\mathrm{supp}(X)=0$. It has been proved (see \cite[Theorem 1.1]{MR2318408}) that the graph distance $\mathcal{D}_{n}(u,v)$ between two uniformly chosen vertices is proportional to $\log\log n$. Hence, these distances are ultra small. When the weights are independent and exponentially distributed, the asymptotic distribution for the minimum weight between two randomly chosen vertices is given by Theorem \ref{dist}, see \cite{BHH10} while in the same paper it is shown that $H_{n}(u,v)$ satisfies the following central limit theorem:
	\be
	\frac{H_{n}(u,v)-\alpha\log n}{\sqrt{\alpha\log n}}\overset{d}{\to}\mathcal{Z},
	\ee
where $\alpha=\frac{2(\tau-2)}{\tau-1}$ and $\mathcal{Z}$ has a standard normal distribution.
In this regard, Theorem \ref{thm::main2} shows the existence of a different behaviour for the hopcount for weight distributions given in (\ref{peso}), and so it shows that there exists different universality classes for the hopcount.
Thus, in the metric space defined by $H_{n}(u,v)$, under the hypothesis of Theorem \ref{thm::main2}, typical distances are small but not ultra-small, meaning that the geometry of the graph significantly changes when switching from one metric to the other.
It would be of interest to investigate the behaviour of $H_{n}(u,v)$ in the explosive setting, and that of $W_{n}(u,v)$ for more general edge weights.

The above results are in sharp contrast to the setting where the degree sequence of the configuration model has finite variance, as investigated in \cite{BHH13} (see also \cite{BHH10} for the setting where the edge weights are exponentially distributed). Indeed, in \cite{BHH13}, Bhamidi, the second author and Hooghiemstra show that when the degree sequence is of finite variance (with an extra logarithmic moment), then first passage percolation has only one universality class in the sense that $W_n(u,v)-\gamma\log{n}$ converges in distribution for some $\gamma>0$, while $H_n(u,v)$ satisfies an asymptotic central limit theorem with asymptotic mean and variance proportional to $\log{n}$ (in \cite{BHH13}, i.i.d.\ degrees are a special case, in the more general case, $\gamma$ and $\alpha$ can depend on $n$). As we see for $\tau\in (2,3)$, the weight distribution has at least two universality classes, depending on whether the age-dependent branching process that approximates the local neighborhoods is explosive or not.

Recently, competition models have attracted considerable attention. In \cite{Deijfen13thewinner},  the spread of two competing infections on the configuration model with power-law exponent $\tau\in (2,3)$ and with i.i.d.\ exponential edge weights have been studied. The result is that one of the infection types will almost surely occupy all but a finite number of vertices. A natural question is whether these results still hold for passage times satisfying the explosive conditions given in this paper.

In the same graph setting, in \cite{KHB13} we have investigated the competition of two competing infections with fixed, but not necessarily equal, speeds.  The faster infection is shown to occupy almost all vertices, while the slower one can occupy only a random subpolynomial fraction of the vertices. More recently, the second and third author show that when the speeds are equal, then {\em coexistence} can occur, in the sense that both types can occupy a positive proportion of the graph \cite{KH15}. It would be of interest to investigate whether this extends to the setting of Theorem \ref{thm::main2}.

\section{Preliminaries}
In this section, we give some preliminaries needed in our proofs.

\subsection{Configuration model and shortest-weight graphs}

We obtain the configuration model $\CMnD$ (see \cite{Bollobas01}) on $[n]=\{1,\dots,n\}$ with i.i.d.\ degree distribution $D$ by the following procedure:\\
\begin{enumerate}
\item[(1)] We assign to each vertex $i\in[n]$  an i.i.d.\ random variable $D_i\sim D$ describing the number of half-edges of that vertex.
\item[(2)] Let $|\CL_n|=\sum_{i=1}^n D_{i}$ be the total number of half-edges in the graph. If $|\CL_n|$ is odd, then we add an extra half-edge to the vertex $n$. Let $\mathcal{L}_{n}=\{h_{1},h_{2},\dots ,h_{|\CL_n|}\}$ be the ordered set of these half-edges.
\item[(3)] We start pairing $h_{1}$ with $h_{i_{1}}$, where $h_{i_{1}}$ is chosen uniformly at random in $\mathcal{L}_{n}\setminus\{h_{1}\}$. At the second step, we pick the first unpaired half-edge $h_{i_{2}}$ and we pair it with a half-edge $h_{i_{3}}$ chosen uniformly at random in $\mathcal{L}_{n}\setminus\{h_{1},h_{i_{2}},h_{i_{3}}\}$. The $k$-th step consists in picking the first half-edge $h_{i_{2k}}$ in $\mathcal{L}_{n}\setminus\{h_{1},h_{i_{2}},h_{i_{3}},h_{i_{4}},\dots ,h_{i_{2k-1}}\}$ and connecting it to a half-edge chosen uniformly at random in $\mathcal{L}_{n}\setminus\{h_{1},h_{i_{2}},h_{i_{3}},h_{i_{4}},\dots ,h_{i_{2k}}\}$.
\item[(Stop)] The process ends when there are no more half-edges to pair.
\end{enumerate}
This process allows us to pick an arbitrary half-edge every time we start pairing a half-edge, i.e., the process of pairing is {\em exchangeable}. Hence, we can do this in the same order as the edge-weights require it, to get the vertices that are closest, in terms of weight, to $u$ or to $v$.

We define the \textit{shortest-weight graph} ${\sf SWG}_{m}^{v}$, starting from vertex $v$, as follows. We start with a vertex $v$ and we define ${\sf SWG}_{0}^{v}=\{v\}$. We define  ${\sf SWG}_{m}^{v}$ as the ordered sequence $\{v, v_{1},y_{1}, e_{1}, v_{2},e_{2}, y_{2}, \dots, v_{m},e_{m}, y_{m}\}$, where $v_{i}$ are vertices, $e_{i}$ are edges, and $y_{i}$ is the weight of the edge $y_{i}$ defined inductively as follows:\\
We start with vertex $v$. Then $e_{1}$ is the edge with the minimal weight starting from $v$ and $v_{1}$ is the vertex that is at the other end of the newly paired edge $e_{1}$, while $y_{1}=Y_{e_1}$ is the weight on the edge $e_1$.
In general, $v_{i}$ is the vertex for which
	\be
 	v_{i}=\underset{w\in[n]\setminus\{ v_{1},\ldots,v_{i-1}\}}{\arg\min}W_{n}(v,w).
	\ee
Here we note that the minimum is always attained by a vertex $v_i$ that is a neighbor of a vertex in $\{ v_{1},\ldots,v_{i-1}\}$ that is not in $\{ v_{1},\ldots,v_{i-1}\}$ itself.
Also, $e_{i}$ is the edge that connects $v_{i}$ to one of the vertices in ${\sf SWG}_{i-1}^{v}$ and $y_{i}=Y_{e_{i}}$ its weight.
This process is generally called ``the exploration process'' of the neighborhood of $v$ in the graph, in the context of first-passage percolation on the configuration model it appears for instance in \cite{BHH10}.

\subsection{The weighted graph and the age-dependent branching process}

Our aim in proving Theorem \ref{thm::main1} is to give the weight of the shortest-weight path in terms of the explosion time of an age-dependent branching process. Let $\tilde{B}_{i}$ stand for the number of edges minus 1 that are incident to $v_{i}$. In each step of the growth of $\CMnD$, an arbitrary half-edge is chosen and paired to a uniformly chosen unpaired half-edge. Thus, the probability of picking a half-edge that is incident to a vertex with $j$ other half-edges is proportional to $(j+1)\Pv(D=j)$, and thus we get the size-biased distribution \eqref{eq::sizebiased} as a natural candidate for the forward degrees of the vertices $v_i$ in the exploration process. More precisely, by \cite[Prop.4.7]{BHH10}, the variables $(\widetilde{B}_{i})_{i=1}^{n^{\rho}}$ can be coupled to an i.i.d. sequence of random variables $(B_{i})_{i=2}^{n^{\rho}}$ with the size-biased distribution $B$ given in \eqref{eq::sizebiased}. We cite this proposition for the reader's convenience:\footnote{Be aware of the differences in notation between \cite{BHH10} and this paper. What we call $B_i$ is called $B_i^{\rm \sss(ind)}$ in \cite{BHH10}, and what we call $\widetilde{B}_i$ is called $B_i$ in \cite{BHH10}.}

\begin{proposition}[{\cite[Proposition 2.1]{BHH10}}]\label{prop::coupling}
There exists $0<\rho<1$ such that the random vector $(\widetilde{B}_{m})_{m=2}^{n^{\rho}}$ can be coupled to an independent sequence of random variables $(B_{m})_{m=2}^{n^{\rho}}$ with probability mass function $B$ given in \eqref{eq::sizebiased} and $(\widetilde{B}_{m})_{m=2}^{n^{\rho}}=(B_{m})_{m=2}^{n^{\rho}}$ w.h.p.
\end{proposition}

\begin{proof}
See \cite[Proposition 4.5]{BHH10} and the proof of Proposition 4.7 in \cite[Appendix A.2]{BHH10}.
\end{proof}

An immediate consequence of Proposition \ref{prop::coupling} is that the ${\sf SWG}^{v}_{n^{\rho}}$ is w.h.p.\ a tree.
We now consider a modified age-dependent  process defined as follows:
\begin{itemize}
\item[$\rhd$] Start with the root which dies immediately giving rise to $D$ children.
\item[$\rhd$] Each alive offspring lives for a random amount of time, with distribution $F_{\sss Y}$ independent from any other randomness involved.
\item[$\rhd$] When the $m$th (where $m>1$) vertex dies, it leaves behind $B_{m}$ alive offspring.
The process is a modified age-dependent two-stage process in the sense that the offspring in the first generation is different from the offspring in second and further generations.
\end{itemize}

The construction of the ${\sf SWG}_{n^{\rho}}^{v}$ is equivalent to this construction, but then on the graph $\CMnD$ rather than on the branching process tree. In Theorem \ref{thm::main1}, we assume that $(h_{\sss B},F_{\sss Y})$ is s.t. the process is explosive, where $h_{\sss B}$ is the probability generating function of $B=D^{\star}-1$.

\subsection{Bond percolation on configuration model}\label{bondper}

Bond percolation on any graph is defined as follows (see \cite{Janson:arXiv0804.1656}): we delete every existing edge independently with probability $1-p$. The remaining edges form the percolated graph that we denote by $G_p$.
For the configuration model, this process is equivalent to the following (see \cite[Remark 1.1.]{Janson:arXiv0804.1656}):
we consider every half-edge independently, and we remove it with probability $1-\sqrt{p}$, for a fixed $p$ s.t. $0\leq p\leq 1$, then we connect it with a new vertex with degree $1$. We finally remove these new vertices, and their incident edges. Since an edge consists
of two half-edges, and each survives with probability $\sqrt{p}$, this is equivalent to randomly deleting an edge with probability $1-p$, independently of all other edges, so the two processes are equivalent.
In our case $p=p(\varepsilon_{0})$ will be the probability that the weight of an edge is less than $\varepsilon_{0}$, for an appropriately chosen $\varepsilon_{0}>0$.
So $G_{p(\varepsilon_{0})}$ consists only of edges of length less than $\varepsilon_{0}$.
We consider $G_p$ as a subgraph of $G$ with the same vertices, but with fewer edges.
%where $\pi_{p}(v)$ is the vertex $v$ after percolation.

Janson \cite{Janson:arXiv0804.1656} has shown that the percolated configuration model is equal in distribution to a configuration model with a new degree distribution $D^{p}= \BIN(D,\sqrt{p}),$ except for some extra vertices of degree 1 that are irrelevant to us. In the next lemma, we show that this new degree sequence again satisfies a power law:
\begin{lemma}[Percolation on the CM]
\label{csh}
Fix $p\in(0,1]$. Let $G= \CMnD$ be a configuration model with degree distribution satisfying \eqref{eq::F}. Then, $G_{p}$ can be representeted as a configuration model with degree distribution that again obeys a power-law distribution of the same form of \eqref{eq::F}, but with different constants $c_{1}$ and $C_{1}$.
\end{lemma}

Before giving the proof, we state a useful lemma about concentration of binomial random variables:

\begin{lemma}[Concentration of binomial random variables]
\label{concenthop}
Let $R$ be a binomial random variable.
Then
	\begin{eqnarray}
	\Pv(R\geq 2\mathbb{E}[R])\leq\exp\{-\mathbb{E}[R]/8\},\\
		\Pv(R\leq \mathbb{E}[R]/2)\leq\exp\{-\mathbb{E}[R]/8\}.
	\end{eqnarray}
\end{lemma}
\begin{proof}
See e.g., \cite[Theorem 2.19]{van2009random}.
\end{proof}

\begin{proof}[Proof of Lemma \ref{csh}]

Let $D^{p}$ be the degree distribution for vertices in $G_{p}$, so that $D^p$ has a $\BIN(D,\sqrt{p})$ distribution. We want to show that there exists two constants $\tilde{c_{1}}$ and $\widetilde{C}_{1}$ such that
	\be\label{powergp}
	\frac{\tilde{c}_{1}}{k^{\tau-1}}\leq\Pv(D^{p}>k)\leq\frac{\widetilde{C}_{1}}{k^{\tau-1}}.
	\ee
The upper bound is obvious since
	\be\ba
	\Pv(D^{p}>k)=\Pv(\BIN(D,\sqrt{p})>k)\leq \Pv(D>k)\leq \frac{C_{1}}{k^{\tau-1}}
	\ea\ee
so that $\widetilde{C}_{1}=C_{1}$.
For the lower bound, we first fix a constant $K$ and consider $k\leq K$. We choose $\tilde{c}_{1}=\tilde{c}_{1}(K)$ sufficiently small, so that the lower bound for $k\leq K$ is trivially satisfied. To prove the inequality for $k>K$, we split
	\eqan{
 	\Pv(\BIN(D,\sqrt{p})>k)&\geq \Pv(\BIN(D,\sqrt{p})>k, D\geq \frac{2k}{\sqrt{p}})\\
	&\geq \Pv\left(D\geq \frac{2k}{\sqrt{p}}\right)-\Pv\left(\BIN(D,\sqrt{p})\leq k, D\geq\frac{2k}{\sqrt{p}}\right)\nn\\
	&\geq \frac{c_{1}\sqrt{p}^{\tau-1}}{(2k)^{\tau-1}}-\Pv\left(\BIN(D,\sqrt{p})\leq k, D\geq\frac{2k}{\sqrt{p}}\right).\nn
	}
Now, by stochastic domination of $\BIN(m,\sqrt{p})$ by $\BIN(n,\sqrt{p})$ when $m\leq n$,
	\eqn{
	\Pv\left(\BIN(D,\sqrt{p})\leq k, D\geq\frac{2k}{\sqrt{p}}\right)\leq\Pv\left(\BIN\big(\frac{2k}{\sqrt{p}},\sqrt{p}\big)\leq k\right).
	}
By Lemma \ref{csh},
	\be
	\Pv\left(\BIN(D,\sqrt{p})\leq k, D\geq\frac{2k}{\sqrt{p}}\right)\leq \e^{-k/4}.
	\ee
Thus,
	\be
	\Pv(\BIN(D,\sqrt{p})>k)\geq c_{1}\sqrt{p}^{\tau-1}/(2k)^{(\tau-1)}
	-\e^{-k/4}\geq \frac{\tilde{c}'_{1}}{k^{\tau-1}},
	\ee
so that $\tilde{c}_{1}$ is the minimum between $\tilde{c}_{1}(K)$ and $\tilde{c}'_{1}$.
\end{proof}

A result on the size of the giant component in the configuration model has first been proved by Molloy and Reed (see \cite{Molloy00acritical}). In the context of percolation on the configuration model, we rely on the following theorem by Janson:

\begin{theorem}[{\protect\cite[Proposition 3.1]{Janson:arXiv0804.1656}} in the case of i.i.d.\ degrees]
\label{bondperc}
Let $\CMnD$ have i.i.d. degrees with distribution function satisfying \eqref{eq::F}.
Then w.h.p.\ there is a giant component $\mathcal{C}_{1}$ if and only if $\mathbb{E}D(D-2)>0$.
In particular, its size $v(\mathcal{C}_{1})$ is given by
	\be
	\ba\label{giacom} \frac{v(\mathcal{C}_{1})}{n}\stackrel{\Pv}{\rightarrow} 1-h_{\sss D}(\xi) \ea,
	\ee
where $\xi$ satisfies
	\be
	\ba\label{giacom2} \frac{h'_{\sss D}(\xi)}{\Ev[D]}=\xi \ea,
	\ee
Further,
	\be\label{giadeg}
	\frac{v_{k}(\mathcal{C}_{1})}{n}\buildrel{\Pv}\over{\longrightarrow} \Pv(D=k)(1-\xi^{k}),
	\ee
where $v_{k}(\mathcal{C}_{1})$ is the number of vertices with degree $k$ in the giant component.
\end{theorem}

Note that
	\be\ba
 	\frac{h'_{\sss D}(s)}{\Ev[D]}&=\frac{\sum_{k}ks^{k-1}\Pv(D=k)}{\mathbb{E}[D]}
	=\sum_{k} s^{k}(k+1)\frac{\Pv(D=k+1)}{\mathbb{E}[D]}\\
                           & =\sum_{k}s^{k}\Pv(D^{\star}-1=k)=h_{\sss D^{\star}-1}(s).
                 \ea\ee
Using this, it follows that \eqref{giacom2} is equivalent to
	\be
	h_{\sss D^{\star}-1}(\xi)=\xi.
	\ee
Note that the solution $\xi$ to $h_{\sss D^{\star}-1}(\xi)=\xi$ is the extinction probability of a branching process with offspring distribution $D^{\star}-1$. Further, $1-h_{\sss D}(\xi)$ is the survival probability of a BP where the root has offspring $D$ and all other individuals have offspring distributed as $D^{\star}-1$.
Due to Janson \cite{Janson:arXiv0804.1656} (see also the first paper on percolation on the configuration model by Fontoulakis \cite{2007math......3269F}), percolation on $\CMnD$ has the same distribution as a configuration model with percolated degrees, where we keep each half-edge with probability $\sqrt{p}$, so that the degree of any vertex in the percolated graph is distributed as $\BIN(D, \sqrt{p})$. Hence, the combination of \cite{Janson:arXiv0804.1656}, then Lemma \ref{csh} and
Theorem \ref{bondperc} yields the following corollary:

\begin{corollary}[Giant component of percolated configuration model]
\label{cor2.5}
Fix $p \in [0,1]$, and consider percolation with parameter $p$ on the configuration model with i.i.d.\ degrees having distribution $F_{\sss D}$ satisfying \eqref{eq::F}.  Then, in the percolated graph $G_{p}$, the new degree distibution satisfies \eqref{eq::F} with different coefficients, and the giant component has size $v(\mathcal{C}_{1})$ s.t.
	\be
	\ba
	\label{giacom1}
	\frac{v(\mathcal{C}_{1}(G_{p}))}{n}\stackrel{\Pv}{\rightarrow} 1-h_{\sss D}(\xi(p)),
	\ea
	\ee
where $h_{\sss D}$ is the p.g.f.\ of $D$ and $\xi(p)$ is the extinction probability of a branching process where the root is present only with probability $\sqrt{p}$ and the offspring distribution is
	\be
	\BIN(D^{\star}-1,\sqrt{p}):=B^{p}.
	\ee
\end{corollary}

If we denote the extinction probability of a BP with offspring distribution $B^{p}$ by $\chi(p)$, then it is easy to prove that $\xi(p)=1-\sqrt{p}+\sqrt{p}\chi(p)$.

Now that we have gathered all preliminaries, we are ready to prove our main results.

\section{\emph{Proof of Theorem \ref{thm::main1}}}
In this section, we use the results of Section 2 to prove Theorem \ref{thm::main1}.
We want to prove that given $\varepsilon$ arbitrarily small, w.h.p.
	\be\ba\label{prima}
	V_{n^{\rho}}^{\sss(1)}+V_{n^{\rho}}^{\sss(2)}\leq W_{n}\leq V_{n^{\rho}}^{\sss(1)}+V_{n^{\rho}}^{\sss(2)}+\varepsilon.
	\ea\ee
To prove the lower bound we have to show that, for a proper choice of $\rho$, ${\sf SWG}_{n^{\rho}}^{u}$, ${\sf SWG}_{n^{\rho}}^{v}$ are w.h.p.\ disjoint. Once we know that they are disjoint, $V_{n^{\rho}}^{\sss(1)}$ and $V_{n^{\rho}}^{\sss(2)}$ denote the time to reach the $n^{\rho}$th individuals in the clusters, hence, the result. That disjointedness is true for a proper choice of $n^{\rho}$ follows from the following proposition
(see \cite[Proposition 4.7]{BHH10}).

\begin{proposition}[Disjointedness of SWGs]
\label{non interference}
There exists a $\rho>0$ such that, w.h.p.
\be
\mathcal{V}({\sf SWG}^{u}_{n^{\rho}})\cap \mathcal{V}({\sf SWG}^{v}_{n^{\rho}})=\varnothing.
\ee
where $\mathcal{V}{\sf (SWG)}$ is the set of vertices of the shortest-weight graph.
\end{proposition}

\begin{proof}
The proof of the proposition follows directly from \cite[Lemma 2.2]{KHB13}, which in turn follows from \cite[Proposition 4.7]{BHH10}.
\end{proof}

Proposition \ref{non interference} immediately proves that $W_{n}(u,v)\geq V_{n^{\rho}}^{\sss(1)}+V_{n^{\rho}}^{\sss(2)}$. To prove the corresponding upper bound, we can decompose the shortest-weight path as the union of three components, the first lies in ${\sf SWG}^{u}_{n^{\rho}}$, the second in ${\sf SWG}^{v}_{n^{\rho}}$ and the third is the minimal weight path that connects these two clusters. We will show that the upper bound in \eqref{prima} holds w.h.p.,
%\be
%\label{upperbound}W_{n}\leq V_{n^{\rho}}^{\sss(1)}+V_{n^{\rho}}^{\sss(2)}+\varepsilon,
%\ee
where $\varepsilon$ bounds the weight of the minimal connecting path from above. The bound in \eqref{prima} is a consequence of the following proposition:

\begin{proposition}[Small-weight connection between SWGs]
\label{shortweight}
For any fixed $\varepsilon>0$, w.h.p. there exists a path that connects ${\sf SWG}_{n^{\rho}}^{u}$ and ${\sf SWG}_{n^{\rho}}^{v}$ having weight less than $\varepsilon$.
\end{proposition}
The proof of Proposition \ref{shortweight} consists of a combination of the bond percolation methods described in Section 2.3 and a layering decomposition of the percolated graph. This layering decomposition is also useful in the case of the unpercolated graphs, see e.g. \cite{KHB13}. We keep ${\sf SWG}^{u}_{n^{\rho}}$ and ${\sf SWG}^{v}_{n^{\rho}}$ and delete every other edge with probability $\Pv(X>\varepsilon)$, where $Y=1+X$ is the weight of the edge. Then we decompose the percolated graph in the following sets of vertices or layers:
	\be
	\label{def::Gamma_i}
	\Gamma_i^{p}:=\{ v\in G_{p}\colon D_{v}^{p}>u_i \},
	\ee
where $u_i$ is defined recursively by
	\be
	\label{eq::ui_recursion}
	u_{i+1} =\left(\frac{u_{i}}{C\log n}\right)^{1/(\tau-2)},
	\quad\quad u_0:= {n^{\rho_0}},
	\ee
where $\rho_0<\rho(\tau-2)$ and $\rho$ is defined in Proposition \ref{shortweight}. A simple calculation yields that
	\be
	\label{def::ui}
	u_i= n^{\rho_0 ((\tau-2)^{-i})} (C\log n)^{-e_i}
	\qquad\mbox{ with } \qquad
	e_i =  \frac{1}{3-\tau}\left( \left( \frac{1}{\tau-2}\right)^{i}-1\right).
	\ee
Also, we define $v^{\star}$ as the maximum degree vertex of the graph, if there are more we choose one uniformly at random. By \eqref{eq::F}, a lower bound on the (percolated) degree of $v^{\star}$ follows:
	\be\label{lowbo}
	\lim_{n\rightarrow\infty}\Pv\left(\max_{i\in [n]}D_{i}^{p}<\left(\frac{n}{b\log n}\right)^{1/(\tau-1)}\right)=0
	\ee	
for an appropriate costant $b$.
So, w.h.p., $D_{v^{\star}}^p>(\frac{n}{\log n})^{1/(\tau-1)}$.
The following lemma (see \cite[Lemma 3.3]{KHB13}) describes how these layering sets are connected in $G_{p}$:

\begin{lemma}[Connectivity lemma]\label{lem::gamma_i_connectivity}
With $u_i$ and $\Gamma_{i}^{p}$ defined as in \eqref{def::ui} and \eqref{def::Gamma_i}, for every $v\in \Gamma_i^{p}$, w.h.p.\ there is a vertex $w\in \Gamma_{i+1}^{p}$ such that $(v,w)\in \CE(G_{p})$, where $\CE(G_{p})$ is the set of edges in $G_{p}$.
Furthermore, w.h.p.\ the previous statement can be applied repeatedly to build a path from $\Gamma_0$ to $\Gamma_{i}$ as long as
	\be
	\label{eq::ibound}
	i < -\frac{\log((\tau-1)\rho_0)}{|\log(\tau-2)|}.
	\ee
\end{lemma}

We want to prove the existence of a path between ${\sf SWG}_{n^{\rho}}^{u}$ and ${\sf SWG}_{n^{\rho}}^{v}$ of arbitrary small length. By Lemmas \ref{csh} and \ref{lem::gamma_i_connectivity}, it follows that the connectivity lemma is still valid for $G_{p}$. We still need to check if ${\sf SWG}_{n^{\rho}}^{u}$ intersects with $\Gamma_{0}^{p}$. The next lemma handles this:

\begin{lemma}[Intersection of ${\sf SWG}_{n^{\rho}}^{u}$ and $\Gamma_{0}^{p}$]
\label{lem-inters}
With high probability,  ${\sf SWG}_{n^{\rho}}^{u}\cap\Gamma_{0}^{p}\neq\varnothing$ as $n\rightarrow \infty$.

\end{lemma}

\begin{proof} We call $V_{M_{n^{\rho}}}$ the vertex having the maximum percolated degree in ${\sf SWG}_{n^{\rho}}^{u}$, that is, the maximal degree vertex after percolation.
Then its degree $\deg{V_{M_{n^{\rho}}}}$ is given by
	\be
	\deg{V_{M_{n^{\rho}}}}=\max_{1\leq i\leq n^{\rho}}\BIN(B_{i},\sqrt{p}),
	\ee
where $B_{i}$ is the $i$th vertex chosen in the growth of ${\sf SWG}_{n^{\rho}}^{u}$, so that by \cite[Proposition 2.1]{KHB13} $B_{i}$ are i.i.d.\ random variables with distribution from \eqref{eq::F}. By Lemma \ref{csh} the same is true for the variables $\BIN(B_{i},\sqrt{p})$ with a different constant in \eqref{eq::F}.
So, elementary calculations yield that for any constant $r>0$ and some proper constant $\tilde{c}_{p}>0$
	\be\label{maxak}
	\Pv\Big(\deg{V_{M_{n^{\rho}}}}<\Big(\frac{n^{\rho}}{r}\Big)^{1/(\tau-2)}\Big)
	<\e^{-\tilde{c}_{p}r}.
	\ee
Choosing $r=c\log n$ establishes the claim.
 \end{proof}

As a result of Lemma \ref{lem-inters}, w.h.p.\ $u$ is connected to some vertex $\tilde{u}$ in $\Gamma_{0}^{p}\cap {\sf SWG}_{n^{\rho}}^{u}$. Then, we use the Connectivity Lemma \ref{lem::gamma_i_connectivity} in $\mathcal{C}_{1}(G_{p})$ to construct a path from $\tilde{u}$ to $\Gamma_{i*}^{p}$, where $i^{\star}$ is the last index when $\Gamma_{i}^{p}$ is w.h.p.\ non-empty. Finally, Lemma \ref{lc} below shows that we can connect this path in less than two steps with the maximum degree vertex $v^{\star}_{1}$ in $G_{p}$. We next perform the details of this proof. In it, we will use the next lemma, which is \cite[Lemma 4.1]{KHB13}:

\begin{lemma}[Direct connectivity lemma]
\label{connectset}
Consider two sets of vertices $A$ and $B$. If the number of half-edges $H_{\sss A}=o(n)$ and $H_{\sss B}$ satisfying
	\be
	\frac{H_{\sss A}H_{\sss B}}{n}>C(n),
	\ee
then, conditioning on the event $\{|\mathcal{L}_{n}|<2\mathbb{E}[D]n\}$, with $N(B)$ denoting the neighbors of $B$,
	\be
	\Pv(A\cap N(B)=\varnothing)< \exp\{-C(n)/(4\mathbb{E}[D])\}.
	\ee
\end{lemma}
%\begin{proof}
%See \cite[Lemma 3.2]{KHB13}.
%\end{proof}

\begin{lemma}[Two-hop connection to maximum degree vertex]
\label{lc}
Let $i^{\star}$ be the last $i$ for which $\Gamma_{i*}^{p}\neq \varnothing$. If $v\in\Gamma_{i^{\star}}^{p}$, then there exists $w\in\Gamma_{i^{\star}}^{p}$ s.t. $(v,w),(w,v^{\star})\in \CE(G)$, with  $v^{\star}$ maximum degree vertex in the percolated graph.
\end{lemma}

\begin{proof}

We define $\Gamma_{\varepsilon}^{p}=\{u\in G_{p} \text{ s.t.\ }D_u^p> n^{\frac{\tau-2}{\tau-1}+\varepsilon} \}$, we denote the number of half-edges in the graph that are connected to vertices with degree $\geq y$ by $\mathcal{S}_{\geq y}^{p}$.
Then,
	\be\ba
	\mathcal{S}^{p}_{\geq y}\stackrel{d}{=}
	\sum_{i=1}^{n}D_{i}^{p}\mathbbm{1}_{\{D_{i}^{p}\geq y\}},
	\ea\ee
where $D_{i}^{p}\overset{d}{=}\BIN(D_{i},\sqrt{p})$, with $D_{i}$ the degree of $i\in[n]$.
Therefore,
	\be
	\mathcal{S}^{p}_{\geq y}\geq y\sum_{i=1}^{n}\mathbbm{1}_{\{D_{i}^{p}\geq y\}}.
	\ee
Let $	R=\sum_{i=1}^{n}\mathbbm{1}_{\{D_{i}^{p}\geq y\}}$, so that $R\sim \BIN(n,\tilde{p})$, where $\tilde{p}=\Pv(D_{i}^{p}\geq y)$. Then, $\mathbb{E}[R]=n\tilde{p}$.
Thus, using Lemma \ref{concenthop}, with  $\tilde{p}\geq\tilde{c}_{1}/y^{\tau-1}$,
	\be
	\Pv\left(\mathcal{S}^{p}_{\geq y}<\frac{1}{2}n\frac{\tilde{c}_{1}}{y^{\tau-1}}\right)\leq\Pv\left(\mathcal{S}^{p}_{\geq y}<\frac{1}{2}\mathbb{E}[R]\right)<\exp\{-\mathbb{E}[R]/8\}.
	\ee
Then, w.h.p.
	\be\label{Sy}
	\ba\mathcal{S}^{p}_{\geq y}\geq \frac{1}{2}n\frac{c_{1}}{(\tau-2)}y^{(2-\tau)}\ea.
	\ee
From \eqref{eq::ui_recursion} and \eqref{lowbo}, if $v\in\Gamma_{i^{\star}}^{p}$ then $D_v^p>(\frac{n}{(\log n)^{\alpha}})^{(\frac{\tau-2}{\tau-1})}$ w.h.p.\ for a certain $\alpha>0$.
We can apply \eqref{Sy} with $y=n^{\frac{\tau-2}{\tau-1}+\varepsilon}$ to see that the number of half-edges in $\Gamma_{\varepsilon}^{p}$ is w.h.p. at least $(nc_{1}/(\tau-2))n^{-(\tau-2)(\frac{\tau-2}{\tau-1}+\varepsilon)}$, where the exponent of $n$ equals $1-(\tau-2)(\frac{\tau-2}{\tau-1}+\varepsilon)>0$ when $\varepsilon>0$ is sufficiently small.

Applying Lemma \ref{connectset} to $H_{\{v\}}$ and $H_{\Gamma_{\varepsilon}^{p}}$ from \eqref{Sy} we see that there exists a vertex $w\in\Gamma_{\varepsilon}^{p}$ such that $(v,w)$ forms an edge.
Finally, we apply Lemma \ref{connectset} with $H_{\{w\}}$ and $H_{\{v^{\star}\}}$ to see that w.h.p.\ there is an edge between $w$ and $v^{\star}$.
\end{proof}
By applying Lemma \ref{lc} twice, we get that any pair of vertices in the most external layer $\Gamma_{0}^{p}$ has a path connecting them of length at most $2(i^{\star}+2)$ and weight at most $2(i^{\star}+2)\varepsilon$. Note that $i^{\star}$ does not grow with $n$. This proves \eqref{prima}.

\section{Proof of Theorem \ref{thm::main2}}
In this section, we prove Theorem \ref{thm::main2}, under the assumptions of i.i.d. edge weights $Y=1+X$  with $\inf\mathrm{supp}(X)=0.$
In Section \ref{sec14} we have reduced it to the proof of Proposition \ref{prop-Wn-UB}.

\subsection{Proof of Proposition \ref{prop-Wn-UB}}
We denote by $\mathcal{P}_{W}^{n}(u,v)$ and $\mathcal{P}_{\mathcal{D}}^{n}(u,v)$ the path from $u$ to $v$ with the minimal weight and the one with the minimal number of edges in $\CMnD$. By definition (\ref{Hopcount}), the hopcount is such that $H_{n}(u,v)=|\mathcal{P}_{W}^{n}(u,v)|$. When $u$ and $v$ are in the same connected component,
	\be
	\label{uineq1}
	\mathcal{D}_{n}(u,v)\leq H_{n}(u,v)\leq W_{n}(u,v),
	\ee
where $\mathcal{D}_{n}(u,v)$ is the graph distance between $u$ and $v$.
We prove Proposition \ref{prop-Wn-UB} by finding an upper bound on $W_{n}(u,v)$.

\begin{proof}[Proof of Proposition 1.10]
Let $p=p(\varepsilon_{0})$ be the survival probability of an edge, i.e., $\Pv(X\leq\varepsilon_{0})=p(\varepsilon_{0})$, where $\varepsilon_{0}$ will be chosen later in the proof. We define the $k$-neighborhood of $u$ in a graph $G$ by
	\be
 	\mathcal{N}_{k}(u)=\{\tilde{u}\in G\colon \mathcal{D}_{n}(u,\tilde{u})\leq k\}.
	\ee
We prove the following lemma:
\begin{lemma}[Intersection of first layer and giant percolated component]
For $G=\CMnD$, and any $p>0$,
	\be
	\label{intersection}
	\lim_{k\rightarrow\infty}\liminf_{n\rightarrow\infty}\Pv(\mathcal{N}_{k}(u)
	\cap \mathcal{C}_{1}(G_{p})\neq\varnothing)=1.
	\ee
\end{lemma}

We stress here that $\mathcal{N}_{k}(u)$ is the neighborhood in the unpercolated configuration model.

\begin{proof}
Let $A_{k}=\partial\mathcal{N}_{k}(u)=\{\tilde{u}\in G\colon \mathcal{D}_{n}(u,\tilde{u})=k\}$.
We denote by $V_{M_{k}}$ the vertex with the maximum percolated degree $D_v^d$ in $A_{k}$. By Proposition \ref{prop::coupling}, $|A_{k}|\stackrel{\Pv}{\rightarrow}\infty$.
The number of outgoing edges from $v_{i}\in A_{k}$ can be coupled to i.i.d.\ random variables $B_{i}\sim B$, with $B$ defined in $\eqref{eq::sizebiased}$, so that
	\be
	\deg{V_{M_{k}}}=\max_{i\in A_{k}}\BIN(B_{i},\sqrt{p}),
	\qquad
	V_{M_{k}}={\arg\max}_{i\in A_{k}}\BIN(B_{i},\sqrt{p}).
	\ee
Using \eqref{powergp}, an elementary calculation shows that for any constant $r>0$ and some constant $\tilde{c}_{p}>0$
	\be
	\label{maxak1}
	\Pv\Big(\deg{V_{M_{k}}}<\big(\frac{|A_{k}|}{r}\big)^{1/(\tau-2)}\Big)<\e^{-\tilde{c}_{p}r}.
	\ee
Let $E_{n,k}=\{V_{M_{k}}\notin\mathcal{C}_{1}(G_{p})\}$ with $G=\CMnD$.
To prove \eqref{intersection} it is enough to show that
	\be
	\lim_{k\rightarrow\infty}\limsup_{n\rightarrow\infty}\Pv(E_{n,k})=0.
	\ee
By the law of total probability,
	\be
	\Pv(E_{n,k})=\sum_{j}\Pv(E_{n,k}|\deg V_{M_{k}}=j)\Pv(\deg V_{M_{k}}=j).
	\ee
By \eqref{giadeg},
	\be
	\frac{v_{j}(G_{p})}{n}-\frac{v_{j}(\mathcal{C}_{1}(G_{p}))}{n}
	\stackrel{\Pv}{\rightarrow} \Pv(D^{p}=j)\xi(p)^{j}.
	\ee
Therefore,
	\begin{align}
	\lim_{k\rightarrow\infty}
	\limsup_{n\rightarrow\infty}\Pv(E_{n,k})
	&=\lim_{k\rightarrow\infty}\limsup_{n\rightarrow\infty}
	\sum_{j}\Pv(E_{n,k}\mid
	\deg V_{M_{k}}=j)\Pv(\deg V_{M_{k}}=j)\\
	&=\lim_{k\rightarrow\infty}\sum_{j}\xi(p)^{j}
	\Pv(\deg V_{M_{k}}=j)=0,\nn
	\end{align}
where, in the last equality we have used \eqref{maxak1} which implies $\deg V_{M_{k}}\stackrel{\Pv}{\rightarrow}\infty$ as $k\rightarrow\infty$.
\end{proof}

By \eqref{maxak}, there exists $\varepsilon_{1}=\varepsilon_{1}(k)$ such that $\Pv(\deg V_{M_{k}}\geq j)>1-\varepsilon_{1}$. Notice that $k=k(\varepsilon_{0}, \varepsilon_{1})$ does not depend on $n$. We denote by $u_{1}$ the vertex of $\partial \mathcal{N}_{k}(u)\cap \mathcal{C}_{1}(G_{p(\varepsilon_{0})})$ having the minimal weight distance from $u$. Such a vertex exists with probability $1-\varepsilon_{1}$. Let $v_{1}$ denote the corresponding vertex for $\partial\mathcal{N}_{k}(v)\cap \mathcal{C}_{1}(G_{p(\varepsilon_{0})})$.
In particular, $u_{1}$ and $v_{1}$ are in the giant component of a configuration model with the same power-law exponent $\tau$. By \cite[Theorem 3.1]{MR2318408}, for every $\varepsilon>0$,
	\be\label{gpdist}
  	\lim_{n\rightarrow\infty}
	\Pv\Big(\mathcal{D}_{n}^{p}(\tilde{u}_{1},\tilde{v}_{1})
	\leq\frac{2(1+\varepsilon)\log\log n)}{|\log(\tau-2)|}
	\Big\vert \mathcal{D}_{n}^{p}(\tilde{u}_{1},\tilde{v}_{1})<\infty\Big)=1,
	\ee
where now $\tilde{u}_{1}$ and $\tilde{v}_{1}$ are two vertices chosen uniformly at random from $[n]$ conditionally on being connected and $\mathcal{D}_{n}^{p}(\tilde{u}_{1},\tilde{v}_{1})$ is the distance between $\tilde{u}_{1}$ and $\tilde{v}_{1}$ in $G_{p}$ with $G=\CMnD$.
In our case, however, the points $u_{1}$ and $v_{1}$ are {\em not} chosen uniformly from the giant component, so we prove a more general statement:

\begin{lemma}[Distances between uniform vertices of fixed degree]
\label{conddist}
For $p\in(0,1)$. If $\mathcal{C}_{1}(G_{p})$ is the giant component of $G_{p}$ with $G=\CMnD$ and the degrees $(D_i)_{i\in [n]}$ are i.i.d.\ random variables whose distribution function $F_{\sss D}$ satisfies \eqref{eq::F}, then, for every $k_{1},k_{2}$,
	\be
	\lim_{n\rightarrow\infty}\Pv\left(\mathcal{D}^{p}_{n}(u_{1},v_{1})\leq\frac{2(1+\varepsilon)\log\log n}{|\log(\tau-2)|}
	\mid u_{1},v_{1}\in \mathcal{C}_{1}(G_{p}),D_{u_{1}}^{p}=k_{1},
	D_{ u_{2}}^{p}=k_{2}\right)=1,
	\ee
where $D_{u_{i}}^{p}$ is the degree of $u_{i}$ in $G_{p}$.
\end{lemma}

\begin{proof}
By \eqref{gpdist},
	\be
	\label{du1du2}
	\lim_{n\rightarrow\infty} \Pv\Big(\mathcal{D}^{p}_{n}(u_{1},v_{1})
	\leq\frac{2(1+\varepsilon)\log\log n}{|\log(\tau-2)|}\Big\vert
	u_{1},v_{1}\in\mathcal{C}_{1}(G_{p})\Big)=1,
	\ee
where $u_{1}$ and $v_{1}$ are chosen uniformly at random in $\mathcal{C}_{1}(G_{p})$.
Further,
	\begin{align}
	\label{prod-tot-prob}
	&\Pv\Big(\mathcal{D}_{n}^{p}(u_{1},v_{1})\leq\frac{2(1+\varepsilon)\log\log n}{|\log(\tau-2)|}
	\Big\vert u_{1},v_{1}\in\mathcal{C}_{1}(G_{p})\Big)\\
	&\qquad= \nonumber
	\sum_{k_{1},k_{2}}\Pv\Big(\mathcal{D}_{n}^{p}(u_{1},v_{1})\leq\frac{2(1+\varepsilon)\log\log n}{|\log(\tau-2)|}
	\Big\vert u_{1},v_{1}\in\mathcal{C}_{1}(G_{p}),D_{u_{1}}^{p}=k_{1},D_{v_{1}}^{p}=k_{2}\Big)\nonumber\\
	&\qquad\qquad\times\Pv\Big(D_{u_{1}}^{p}=k_{1},D_{u_{2}}^{p}=k_{2}|u_{1},v_{1}\in\mathcal{C}_{1}(G_{p})\Big). \nonumber
	\end{align}

We have a series of the form
	\be
	a(n)=\sum_{k_{1},k_{2}}a_{k_{1},k_{2}}(n)b_{k_{1},k_{2}}(n),
	\ee
where $a_{k_{1},k_{2}}(n)$ is the first factor on the right-hand side of \eqref{prod-tot-prob}, while $b_{k_{1},k_{2}}(n)$ is the second.

By \cite[Prop 3.1]{Janson:arXiv0804.1656},
	\eqn{
	b_{k}(n)=\Pv\left(\deg u_{1}=k\mid u_{1}\in\mathcal{C}_{1}(G_{p})\right)
	\rightarrow \frac{\Pv(D^{p}=k)(1-\xi(p)^{k})}{1-h_{\sss D}(\xi(p))},
	}
with $\xi(p)$ as in Corollary \ref{cor2.5}.
Then, using the independence of coupling for two vertices (see \cite[Theorem 1.1]{MR2318408}), \eqref{giadeg} and \eqref{giacom2}, for every $k_1,k_2\geq 0$
	\be
	b_{k_{1},k_{2}}(n)\rightarrow
	\frac{\Pv(D^{p}=k_{1})(1-\xi(p)^{k_{1}})\Pv(D^{p}=k_{2})(1-(\xi(p))^{k_{2}})}
	{\left(1-h_{\sss D}(\xi(p))\right)^{2}}\equiv b_{k_{1},k_{2}}.
	\ee
It is straightforward to check that
	\be
	\sum_{k_{1},k_{2}}b_{k_{1},k_{2}}(n)=\sum_{k_{1},k_{2}}b_{k_{1},k_{2}}=1.
	\ee
Further, by \eqref{gpdist}, $a(n)\rightarrow 1$.

Now we are ready to complete the argument. We would like to show that $\lim_{n\rightarrow\infty} a_{\tilde{k}_{1},\tilde{k}_{2}}(n)=1$ for any $\tilde{k}_{1},\tilde{k}_{2}$.  We argue by contradiction.
Suppose instead that there exist $\tilde{k}_{1},\tilde{k}_{2}$ for which
	\be
	\label{ak1k2-contr}
	\liminf_{n\rightarrow\infty} a_{\tilde{k}_{1},\tilde{k}_{2}}(n)=1-\beta<1.
	\ee
Then,
	\be
	\ba
	\liminf_{n\rightarrow\infty}a(n)
	&\leq \liminf_{n\rightarrow\infty} a_{\tilde{k}_{1},\tilde{k}_{2}}(n)
	b_{\tilde{k}_{1},\tilde{k}_{2}}(n)
	+\sum_{(k_{1},k_{2})\neq(\tilde{k}_{1},\tilde{k}_{2})}b_{k_{1},k_{2}}(n)\\
	&=(1-\beta)b_{k_{1},k_{2}}+1-b_{k_{1},k_{2}}=1-\beta b_{k_{1},k_{2}},\nn
	\ea
	\ee
which leads to a contradiction, since $b_{k_{1},k_{2}}>0$. We conclude that \eqref{ak1k2-contr} cannot hold, so that $\liminf_{n\rightarrow\infty} a_{k_{1},k_{2}}(n)=1$ for every $k_1,k_2$.
\end{proof}

\noindent We continue with the proof of Proposition \ref{prop-Wn-UB}. We want to show that, for $u$ and $v$ uniformly chosen in $G,$ any fixed $\varepsilon$, there exists $n(\delta_{1},\varepsilon)$ s.t. for every $ n>n(\delta_{1},\varepsilon)$
 	\be
	\label{equno}
	\Pv\left(\frac{W_{n}(u,v)}{2\log\log n }\leq\frac{1+\varepsilon}{|\log(\tau-2)|}\right)\geq 1-\delta_{1}.
	\ee
Let $E^{'}_{n}=\{\pi_{p}(\mathcal{N}_{k}(u))\cap\mathcal{C}_{1}(G_{p})\neq\varnothing,\pi_{p}(\mathcal{N}_{k}(v))\cap\mathcal{C}_{1}(G_{p})\neq\varnothing\}$ where $G=\CMnD$.
Then
	\be
	\Pv\Big(\frac{W_{n}(u,v)}{2\log\log n }\leq\frac{1+\varepsilon}{|\log(\tau-2)|}\Big)
	\geq\Pv\Big(\frac{W_{n}(u,v)}{2\log\log n }\leq\frac{1+\varepsilon}{|\log(\tau-2)|}|E_{n}^{'}\Big)\Pv(E^{'}_{n})\\
    \ee
Also, by \eqref{intersection}, if $n\geq n(\delta_{1},\varepsilon)$ is so large that $\Pv(E^{'}_{n})\geq 1-\delta_{1}$ then
	\begin{align}
	\label{split}
	\Pv\left(\frac{W_{n}(u,v)}{2\log\log n }\leq\frac{1+\varepsilon}{|\log(\tau-2)|}\right)&\geq\Pv\left(\frac{W_{n}(u,v)}{2\log\log n}\leq\frac{1+\varepsilon}{|\log(\tau-2)|}\Big|E_{n}^{'}\right)(1-\delta_{1}).
	\end{align}
So, we need to prove that there exists an $n(\delta_{2},\varepsilon)$ s.t. for all $n>n(\delta_{2},\varepsilon)$
	\be
	\Pv\left(\frac{W_{n}(u,v)}{2\log\log n}\leq\frac{1+\varepsilon}{|\log(\tau-2)|}|E_{n}^{'}\right)>1-\delta_{2}.
	\ee
By Lemma \ref{conddist}, for any fixed $\varepsilon_{1}$ it holds w.h.p. the following
\be
\mathcal{D}_{n}^{p}(u_{1},v_{1})\leq\frac{2(1+\varepsilon_{1})\log\log n}{|\log(\tau-2)|},
\ee
where $u_{1},v_{1}$ are two vertices in $\mathcal{C}_{1}(G_{p})$ as in \eqref{gpdist}.
So, we choose $\varepsilon_{1}$ satisfying
\be
(1+\varepsilon_{0})(1+\varepsilon_{1})\leq(1+\varepsilon),
\ee
where $\varepsilon_{0}$ is the threshold length for percolation, i.e., we only keep edges with length $1+\varepsilon_{0}$ in $G_{p}$.
Under the hypothesis that $E_{n}^{'}$ is true, we have:
\be
 \mathcal{D}_{n}(u,v)\leq k+\mathcal{D}_{n}^{p}(u_{1},v_{1})+k.
\ee
Let $Y^{u}_{\leq k}$ be the random variable that describes the weight of the path with the minimal distance $\mathcal{P}_{\mathcal{D}}^{n}(u,u_{1})$, $Y_{\leq k}^{v}$ be the weight  for $\mathcal{P}_{\mathcal{D}}^{n}(v,v_{1})$.
Also $Y_{\leq k}^{u} \stackrel{d}{=}Y_{\leq k}^{v}{\buildrel d\over\leq}\sum_{i=1}^{k}Y_{i}$, where $Y_{i}$ are i.i.d. random variables with law $Y$ and the stochastic dominance is due to the weight-dependent choice of $u_{1}$ and $v_{1}$ if there are more elements in the intersections $\partial\mathcal{N}_{k}(u)\cap\mathcal{C}_{1}(G_{p}) \text{ and }\partial\mathcal{N}_{k}(v)\cap\mathcal{C}_{1}(G_{p})$. Then, for any $k$, w.h.p.,
	\be
	 W_{n}(u,v)\leq
	Y_{\leq k}^{u}+\frac{2(1+\varepsilon_{0})(1+\varepsilon_{1})\log\log n}{|\log(\tau-2)|}
	+Y_{\leq k}^{v}.
	\ee
We fix $k=k(\varepsilon_{0},\delta_{1})$ such that
	\be
	\liminf_{k\rightarrow\infty}\liminf_{n\rightarrow\infty}\Pv(E^{'}_{n})\geq 1-\delta_{1}.
	\ee
Thus, for any fixed $\delta_{1}>0,$ there exists $n>n(k,\varepsilon_{0},\varepsilon_{1},\delta_{2})$ s.t. $\Pv(\frac{W_{n}(u,v)}{2\log\log n }\leq\frac{(1+\varepsilon_{0})(1+\varepsilon_{1})}{|\log(\tau-2)|}|E_{n}^{'})>1-\delta_{2}$.
 So, from \eqref{split}, we have to choose $\delta_{1}$ and $\delta_{2}$ s.t.\
	\be
	(1-\delta_{1})(1-\delta_{2})>1-\delta.
	\ee
This completes the proof of Proposition \ref{prop-Wn-UB}.
\end{proof}

\paragraph{Acknowledgements.}
This work is supported by the Netherlands
Organisation for Scientific Research (NWO) through VICI grant 639.033.806 (EB and RvdH), VENI grant 639.031.447 (JK), the Gravitation {\sc Networks} grant 024.002.003 (RvdH) and the STAR Cluster (JK).

\bibliographystyle{abbrv}
\bibliography{bibliografia}

\begin{thebibliography}{10}

\bibitem{2004PhRvE..69b5103A}
R.~{Albert}, I.~{Albert}, and G.~L. {Nakarado}.
\newblock {Structural vulnerability of the North American power grid}.
\newblock {\em Physical Review E}, 69(2):025103, Feb. 2004.

\bibitem{RevModPhys}
R.~Albert and A.-L. Barab\'asi.
\newblock Statistical mechanics of complex networks.
\newblock {\em Rev. Mod. Phys.}, 74:47--97, Jan 2002.

\bibitem{Barabási200069}
A.-L. Barab{\'a}si, R.~Albert, and H.~Jeong.
\newblock Scale-free characteristics of random networks: the topology of the
  world-wide web.
\newblock {\em Physica A: Statistical Mechanics and its Applications},
  281(1–4):69 -- 77, 2000.

\bibitem{BHH10}
S.~Bhamidi, R.~v.~d. Hofstad, and G.~Hooghiemstra.
\newblock First passage percolation on random graphs with finite mean degrees.
\newblock {\em Ann. Appl. Probab.}, 20(5):1907--1965, 2010.

\bibitem{BHH13}
S.~Bhamidi, R.~v.~d. Hofstad, and G.~Hooghiemstra.
\newblock Universality for first passage percolation on sparse random graphs.
\newblock Preprint 2012.
\newblock Available at arXiv: 1210.6839 [math.PR].

\bibitem{Bollobas01}
B.~Bollob\'as.
\newblock {\em Random Graphs}.
\newblock Cambridge University Press, 2001.

\bibitem{Bornholdt:2003}
S.~Bornholdt and H.~G. Schuster, editors.
\newblock {\em Handbook of Graphs and Networks: From the Genome to the
  Internet}.
\newblock John Wiley \& Sons, Inc., New York, NY, USA, 2003.

\bibitem{MR0483050}
P.~L. Davies.
\newblock The simple branching process: a note on convergence when the mean is
  infinite.
\newblock {\em J. Appl. Probab.}, 15(3):466--480, 1978.

\bibitem{Deijfen13thewinner}
M.~Deijfen and R.~v.~d. Hofstad.
\newblock {The winner takes it all}.
\newblock {\em ArXiv e-prints}, June 2013.

\bibitem{erdHos1976evolution}
P.~Erd{\H{o}}s and A.~R{\'e}nyi.
\newblock On random graphs i.
\newblock {\em Publ. Math. Debrecen}, 6:290--297, 1959.

\bibitem{Faloutsos}
M.~Faloutsos, P.~Faloutsos, and C.~Faloutsos.
\newblock On power-law relationships of the internet topology.
\newblock {\em SIGCOMM Comput. Commun. Rev.}, 29(4):251--262, Aug. 1999.

\bibitem{2007math......3269F}
N.~Fountoulakis.
\newblock Percolation on sparse random graphs with given degree sequence.
\newblock {\em Internet Math.}, 4(4):329--356, 2007.

\bibitem{Grey74}
D.~R. Grey.
\newblock Explosiveness of age-dependent branching processes.
\newblock {\em Z. Wahrscheinlichkeitstheorie und Verw. Gebiete}, 28:129--137,
  1973/74.

\bibitem{Harr63}
T.~Harris.
\newblock {\em The theory of branching processes}.
\newblock Die Grundlehren der Mathematischen Wissenschaften, Bd. 119.
  Springer-Verlag, Berlin, (1963).

\bibitem{van2009random}
R.~v.~d. Hofstad.
\newblock Random graphs and complex networks.
\newblock {\em Available on http://www. win. tue. nl/$\sim$ rhofstad/NotesRGCN.
  pdf}, In\\preparation(2015).

\bibitem{MR2318408}
R.~v.~d. Hofstad, G.~Hooghiemstra, and D.~Znamenski.
\newblock Distances in random graphs with finite mean and infinite variance
  degrees.
\newblock {\em Electron. J. Probab.}, 12:no. 25, 703--766, 2007.

\bibitem{Janson:arXiv0804.1656}
S.~Janson.
\newblock On percolation in random graphs with given vertex degrees.
\newblock {\em Electron. J. Probab.}, 14:no. 5, 86--118, 2009.

\bibitem{2010arXiv1001.2172K}
P.~Kaluza, A.~K{\"o}lzsch, M.~T. Gastner, and B.~Blasius.
\newblock The complex network of global cargo ship movements.
\newblock {\em Journal of The Royal Society Interface}, 2010.

\bibitem{KH15}
J.~Komj\'athy and R.~v.~d. Hofstad.
\newblock Fixed speed competition on the configuration model with infinite
  variance degrees: equal speeds.
\newblock Preprint 2015.
\newblock Available at arXiv: 1503.09046 [math.PR].

\bibitem{KHB13}
J.~Komj\'athy, R.~v.~d. Hofstad, and E.~Baroni.
\newblock Fixed speed competition on the configuration model with infinite
  variance degrees: unequal speeds.
\newblock Preprint 2014.
\newblock Available at arXiv:1408.0475 [math.PR].

\bibitem{Molloy00acritical}
M.~Molloy and B.~Reed.
\newblock A critical point for random graphs with a given degree sequence.
\newblock {\em Random Structures $\&$ Algorithms}, 6(2-3):161--180, 1995.

\bibitem{2003SIAMR}
M.~E.~J. {Newman}.
\newblock {The Structure and Function of Complex Networks}.
\newblock {\em SIAM Review}, 45:167--256, Jan. 2003.

\bibitem{citation}
F.~Radicchi, S.~Fortunato, and A.~Vespignani.
\newblock Citation networks.
\newblock In A.~Scharnhorst, K.~B{\"o}rner, and P.~van~den Besselaar, editors,
  {\em Models of Science Dynamics}, Understanding Complex Systems, pages
  233--257. Springer Berlin Heidelberg, 2012.

\bibitem{wu}
J.~Wu, C.~K. Tse, F.~C. Lau, and I.~W.~H. Ho.
\newblock Analysis of communication network performance from a complex network
  perspective.
\newblock {\em Circuits and Systems I: Regular Papers, IEEE Transactions on},
  60(12):3303--3316, 2013.

\end{thebibliography}

\end{document}